\documentclass[a4paper,10pt]{amsart}
\usepackage[english]{babel} 
\usepackage{amsmath,amssymb,amsthm,amsfonts} 
\usepackage{latexsym}
\usepackage[all]{xy} 
\usepackage{color}

\newcommand{\F}{\mathbb F}

\newcommand{\gal}{\mathrm{Gal}}
\renewcommand{\epsilon}{\varepsilon}

\newcommand{\st}{\mathrm{st}}

\newcommand{\rt}{\mathrm{R}_t}
\newcommand{\cl}{\mathrm{Cl}}

\newcommand{\G}{\mathcal G}
\newcommand{\oo}{\mathcal O}
\newcommand{\W}{\mathcal W}

\theoremstyle{plain}
\newtheorem{theorem}{Theorem}[section]
\newtheorem{proposition}[theorem]{Proposition}

\theoremstyle{definition}
\newtheorem{definition}[theorem]{Definition}

\title[On a question on $A$-groups]{Answer to a question on $A$-groups, arisen from the study of Steinitz classes}
\author{Alessandro Cobbe and Maurizio Monge}

\begin{document}
\maketitle

\begin{section}*{Abstract}
  In this short note we answer to a question of group theory from \cite{Cobbe1}. In that paper the
  author describes the set of realizable Steinitz classes for so-called $A'$-groups of odd order,
  obtained iterating some direct and semidirect products. It is clear from the definition that
  $A'$-groups are solvable $A$-groups, but the author left as an open question whether the converse
  is true. In this note we prove the converse when only two prime numbers divide the order of the
  group, but we show it to be false in general, producing a family of counterexamples which are
  metabelian and with exactly three primes dividing the order. Steinitz classes which are realizable
  for such groups in the family are computed and verified to form a group.
\end{section}

\begin{section}{Introduction}
  Let $K/k$ be an extension of number fields with rings of integers $\oo_K$ and $\oo_k$
  respectively. Then there exists an ideal $I$ of $\oo_k$ such that
\[\oo_K\cong \oo_k^{[K:k]-1}\oplus I\]
as $\oo_k$-modules and the ideal $I$ is determined up to principal ideals. Its class in the ideal
class group $\cl(\oo_k)$ of $\oo_k$ is called the Steinitz class of the extension and is denoted by
$\st(K/k)$. For a fixed number field $k$ and a finite group $G$ one can consider the set of classes
which arise as Steinitz classes of tame Galois extensions with Galois group $G$, i.e. the set
\[\rt(k,G)=\{x\in\cl(k):\ \exists K/k\text{ tame Galois, }\gal(K/k)\cong G, \st(K/k)=x\}.\]
A description of $\rt(k,G)$ is not known in general, but there are a lot of results for some
particular groups. These results lead to the conjecture that $\rt(k,G)$ is always a subgroup of the
ideal class group, which however has not been proved in general.  In \cite{Cobbe1} the author
defines $A'$-groups in the following way and proves the above conjecture for all $A'$-groups of odd
order.
\begin{definition}\label{A'groups}
We define $A'$-groups inductively:
\begin{enumerate}
\item
\label{rule1}
 Finite abelian groups are $A'$-groups.
\item
\label{rule2}
 If $\G$ is an $A'$-group and $H$ is finite abelian of order prime to that of $\G$, then
  $H\rtimes_\mu \G$ is an $A'$-group, for any action $\mu$ of $\G$ on $H$.
\item 
\label{rule3}
If $\G_1$ and $\G_2$ are $A'$-groups, then $\G_1\times \G_2$ is an $A'$-group.
\end{enumerate}
\end{definition}
Clearly (see \cite[Proposition 1.2]{Cobbe1}) every $A'$-group is a solvable $A$-group, while it
was asked whether the converse is true. In this short note we find a family of counterexamples for this. In the
last section we show how the techniques from \cite{Cobbe1} can be applied also to the calculation of
the realizable Steinitz classes for these groups, showing in particular that $\rt(k,G)$ is still a
subgroup of the ideal class group, confirming the general conjecture.
\end{section}

\begin{section}{Solvable $A$-groups which are not $A'$-groups}

We start showing a positive result when only two primes divide the order. See
  \cite{taunt1949groups,huppert1967endliche} for general results about the $A$-groups.

\begin{proposition}
\label{prop1}
  An $A$-group $G$ having order divisible by at most two different primes is an $A'$-group.
\end{proposition}

\begin{proof}[Proof]
  Indeed, let $G$ be an $A$-group with order divisible only by the primes $p$ and $q$; it is always
  solvable by Burnside Theorem. By Hall-Higman Theorem \cite[Satz VI.14.16]{huppert1967endliche} a
  solvable $A$-group has derived length at most equal to the number of distinct prime divisors of
  the order, so in our case $G$ has derived length at most $2$ and $G'$ is abelian. If the derived
  length is $1$ then $G$ is abelian, so we are reduced to consider the case of derived length
  exactly $2$.

  We will consider the unique subgroup $K_p$ such that $K_pG'/G'$ is the $p$-Sylow of
  $G/G'$ and $K_p\cap{}G'$ is the $q$-Sylow of $G'$ and we will show it to be normal in $G$. Further by
  Schur-Zassenhaus Theorem it is an
  $A'$-group, being the semidirect product of an abelian $q$-group by an abelian $p$-group. Constructing analogously $K_q$, with $p$ an $q$ flipped, we have that
  $K_p\cap{}K_q=1$, while $K_pK_q$ is all of $G$, so $K_p$ and $K_q$ are direct factors of $G$, since they are normal. Therefore $G$ is isomorphic to $K_p\times{}K_q$ and consequently $G$
  is an $A'$-group by rule \ref{rule3}.

  To construct $K_p$ let's quotient out the $q$-Sylow $S_q$ of $G'$, obtaining the group
  $\tilde{G}=G/S_q$. Its $p$-Sylow, $\tilde{P}$ say, is clearly normal being the inverse image of
  the $p$-Sylow of $G/G'$, which is a $p$-group since we killed all the $q$-part of $G'$. So we have
  the exact sequence
  \[
    1 \rightarrow \tilde{P} \rightarrow \tilde{G} \rightarrow \tilde{G}/\tilde{P} \rightarrow 1,
  \]
  and furthermore $\tilde{G}'$ is equal to $G'/S_q$ being $S_q\subseteq{}G'$, and is contained in
  $\tilde{P}$ being $\tilde{G}/\tilde{P}$ abelian.

  Now $\tilde{G}'$ has a complementary factor in $\tilde{P}$ which is invariant under the action by
  conjugation of the $q$-group $\tilde{G}/\tilde{P}$ by \cite[Theorem 2.3,
    Chap. 5]{gorenstein1980finite}, so let's assume $\tilde{P}=\tilde{G}'\times{}F_p$ say. Clearly
  $F_p$ is a $p$-group which is normal in $\tilde{G}$, and $F_p\tilde{G}'/\tilde{G}'$ is the
  $p$-Sylow of $\tilde{G}/\tilde{G}'={}G/G'$. So if we put $K_p$ to be the preimage of $F_p$ under
  the projection $G\rightarrow{}\tilde{G}$ we have that $K_p$ is normal in $G$, $K_pG'/G'$ is
  the $p$-Sylow of $G/G'$, and $K_p\cap{}G'$ is the $q$-Sylow $S_q$ of $G'$, being the preimage of
  $F_p\cap{}\tilde{G}'=1$.
\end{proof}

For any triple $p,q,r$ of distinct primes we construct now a counterexample which is a metabelian
group.  For any integer $n$ let $C_n$ be the cyclic group on $n$ elements.

Let $a,b$ be integers such that
\[
    qr\mid{}p^a-1,\qquad pr\mid{}q^b-1,
\]
or equivalently such that $\mathrm{ord}_{qr}^\times(p)\mid{}a$ and
$\mathrm{ord}_{pr}^\times(q)\mid{}b$. Let $\F_{p^a}$ and $\F_{q^b}$ respectively be the fields with
$p^a$ and $q^b$ elements, then the multiplicative groups $\F_{p^a}^\times$ and $\F_{q^b}^\times$ act
naturally as automorphisms on the additive groups $\F_{p^a}^+$ and $\F_{q^b}^+$. If
$\phi:C_q\hookrightarrow{}\F_{p^a}^\times$ and $\psi:C_p\hookrightarrow{}\F_{q^b}^\times$ are
embeddings we can consider the semidirect products
\[
  H_1=\F_{p^a}^+\rtimes_\phi{}C_q, \qquad H_2=\F_{q^b}^+\rtimes_\psi{}C_p.
\]
Let's also consider embeddings $\rho_1:C_r\hookrightarrow{}\F_{p^a}^\times$ and
$\rho_2:C_r\hookrightarrow{}\F_{q^b}^\times$, since $\F_{p^a}^\times$ and $\F_{q^b}^\times$ are
abelian groups the actions induced by $C_r$ on $\F_{p^a}^+$ and $\F_{q^b}^+$ commute with those of
$C_q$ and $C_p$, so $\rho_1$, $\rho_2$ induce an action of $C_r$ on $H_1$ and $H_2$ which is trivial
on $C_p$ and $C_q$.

We define
\[
   G = (H_1\times{}H_2) \rtimes_{\rho_1,\rho_2} C_r,
\]
where $C_r$ acts on $H_i$ via $\rho_i$, for $i=1,2$.

\begin{proposition}
$G$ is a metabelian $A$-group which is not an $A'$-group.     
\end{proposition}

\begin{proof}[Proof]
  Indeed, $G$ is metabelian because $\F_{p^a}^+\times\F_{q^b}^+$ is a normal abelian subgroup with
  abelian quotient, isomorphic to $C_q\times{}C_p\times{}C_r$.

  To show that $G$ cannot be obtained applying rule \ref{rule2} in the inductive definition of the
  $A'$-groups we prove that no Sylow subgroup is normal. Since $(r,p)=1$, a $p$-Sylow $P$ is
  contained in $H_1\times{}H_2$, and if normal then $H_2\cap{}P$ would be normal in $H_2$ too, but
  $C_p$ in $\F_{q^b}^+\rtimes{}C_p$ is clearly not normal or it would be complemented by the normal
  subgroup $\F_{q^b}^+$ and $H_2$ would be abelian, which is not the case. The same holds for the
  $q$-Sylow of $H_1$, and similarly $C_r$ cannot be normal unless $G=(H_1\times{}H_2)\times{}C_r$
  and all elements of order $r$ would be contained in the center of $G$, which is not the case.

  To conclude we just need to show that $G$ is not a direct product, so it also cannot be obtained
  applying rule \ref{rule3}. Suppose $G=G_1\times{}G_2$, then exactly one of $G_1,G_2$ has order
  divisible by $r$, so assume $r\mid{} |G_1|$, and we have that $G_1$ contains all $r$-Sylow
  subgroups, so in particular $C_r\subset{}G_1$. Then $G_2$ is contained in the centralizer of
  $C_r$, that considering the definition of $G$ we can see to be equal to
  $C_p\times{}C_q\times{}C_r$. But $r\nmid |G_2|$, and if $p\mid{}G_2$ we would have
  $C_p\subset{}G_2$ and $C_p$ would be the $p$-Sylow, and hence a characteristic subgroup, of $G_2$,
  and consequently normal in $G$, which is absurd. Since we can prove similarly that $q\nmid |G_2|$
  we obtain $G_2=1$.
\end{proof}

We remark that some of the smallest counterexamples are those obtained putting the $(p,q,r;a,b)$
equal to $(5,2,3;2,4)$ and $(13,3,2;1,3)$. The groups produced have orders respectively $12000$ and
$27378$, and are already a bit too far away to be found in a brute-force computer search, as was
performed by the author of \cite{Cobbe1}.

\end{section}

\begin{section}{Realizable Steinitz classes}
In \cite{CaputoCobbe}, for all number fields $k$ and all finite groups $G$, a subgroup $\W(k,G)$ of the ideal class group $\cl(k)$ of $k$ was defined. In \cite[Theorem 2.10]{CaputoCobbe} it has been shown that
\[\rt(k,G)\subseteq\W(k,G)\]
and that there is an equality whenever $G$ is an $A'$-group of odd order (\cite[Theorem 4.3]{CaputoCobbe}). So it is a natural question to investigate whether the equality holds for the solvable $A$-groups constructed above, which are not $A'$-groups, when $p,q,r$
are all odd prime numbers.

\begin{proposition}
Let $p,q,r$ be odd prime numbers, let $G$ be defined as in the previous section and let $k$ be a number field. Then
\[\rt(k,G)=\W(k,G).\]
\end{proposition}

\begin{proof}[Proof]
As we have said above the inclusion
\[\rt(k,G)\subseteq\W(k,G)\]
is true in general and is proved in \cite[Theorem 2.10]{CaputoCobbe}. To show the opposite one we will rely on the notation and the main results of \cite{CaputoCobbe}.

We note that $G$ can be written as a semidirect product of the form $H\rtimes \G$, where $H=\F_{p^a}^+\times\F_{p^b}^+$ and $\G=C_p\times C_q\times C_r$; let $\pi:G\to\G$ be the usual projection. Hence, by \cite[Theorem 3.5]{CaputoCobbe} and \cite[Proposition 4.3]{CaputoCobbe} (applied to $\G$), we obtain
\[\rt(k,G)\supseteq\W(k,\G)^{\#H}\prod_{\ell |\#H}\prod_{\tau\in H\{\ell\}^*}W(k,E_{k,G,\tau})^{((\ell-1)/2)(\#G/o(\tau))},\]
So it suffices to show that
\begin{equation}\W(k,G)\subseteq\W(k,\G)^{\#H}\prod_{\ell |\#H}\prod_{\tau\in H\{\ell\}^*}W(k,E_{k,G,\tau})^{((\ell-1)/2)(\#G/o(\tau))}.\label{tobeproved}\end{equation}

For any prime number $\ell$ dividing $\#G$, the $\ell$-Sylow subgroups of $G$ have exponent $\ell$, i.e. for all $\tau\in G\{\ell\}^*$, the order of $\tau$ is exactly $\ell$.

So let $\tau\in G$ be of order $\ell$. Then we have two possibilities:
\begin{enumerate}
\item[(a)] $\pi(\tau)$ is of order $\ell$. Then for any element $\sigma$ of the normalizer of $\tau$, we have $\sigma\tau\sigma^{-1}=\tau^i$ for some $i$. Hence also $\pi(\sigma)\pi(\tau)\pi(\sigma)^{-1}=\pi(\tau)^i$ and, since $\G$ is abelian, we can conclude that $i=1$. Therefore the normalizer of $\tau$ is equal to its centralizer and so from the definition of $E_{k,G,\tau}$ given in \cite{CaputoCobbe} it is clear that $E_{k,G,\tau}=k(\zeta_\ell)$. Therefore we easily obtain
\[W(k,E_{k,G,\tau})^{((\ell-1)/2)(\#G/\ell)}\subseteq\W(k,\G)^{\#H}.\]
\item[(b)] $\pi(\tau)=1$. In this case $\tau\in H$ and we clearly have
\[W(k,E_{k,G,\tau})^{((\ell-1)/2)(\#G/\ell)}=W(k,E_{k,G,\tau})^{((\ell-1)/2)(\#G/o(\tau))}.\]
\end{enumerate}
So in any case we have shown that $W(k,E_{k,G,\tau})^{((\ell-1)/2)(\#G/\ell)}$ is contained in the subgroup on the right-hand side of the inclusion (\ref{tobeproved}), which is therefore proved, recalling the definition of $\W(k,G)$.
\end{proof}
In particular this proves that $\rt(k,G)$ is a group. It is also straightforward to verify that $G$ is very good, according to the definition given in \cite{CaputoCobbe}.

\end{section}

\renewcommand\refname{References}


\begin{thebibliography}{99} 

\bibitem{CaputoCobbe}
L.~Caputo, A.~Cobbe.
\newblock An explicit candidate for the set of Steinitz classes of tame Galois extensions with fixed Galois group of odd order.
\newblock {\em Proc. London Math. Soc.}, 2013, doi:10.1112/plms/pds067.


\bibitem{Cobbe1}
A.~Cobbe.
\newblock Steinitz classes of tamely ramified galois extensions of algebraic
  number fields.
\newblock {\em Journal of Number Theory}, 130(5):1129 -- 1154, 2010.

\bibitem{gorenstein1980finite}
D.~Gorenstein.
\newblock {\em Finite groups}.
\newblock Chelsea Pub. Co., 1980.

\bibitem{huppert1967endliche}
B.~Huppert.
\newblock {\em Endliche Gruppen I}, volume 134 of {\em Die Grundlehren der
  mathematischen Wissenschaften in Einzeldarstellungen}.
\newblock Springer, 1967.

\bibitem{taunt1949groups}
D.~Taunt.
\newblock On $a$-groups.
\newblock In {\em Mathematical Proceedings of the Cambridge Philosophical
  Society}, volume~45, pages 24--42. Cambridge Univ Press, 1949.

\end{thebibliography}
\end{document}